\documentclass[12pt,twosided,reqno]{amsart}
\usepackage{amscd,amssymb}
\usepackage[all]{xy}
\usepackage{graphicx}
\usepackage{hhline}

\newtheorem{theorem}{Theorem}[section]
\newtheorem{corollary}[theorem]{Corollary}
\newtheorem{lemma}[theorem]{Lemma}
\newtheorem{prop}[theorem]{Proposition}

\theoremstyle{definition}
\newtheorem{definition}[theorem]{Definition}
\newtheorem{example}[theorem]{Example}
\newtheorem{remark}[theorem]{Remark}

\DeclareMathAlphabet{\pazocal}{OMS}{zplm}{m}{n}

\def\dot{\mathchar"013A}
\newcommand{\hdot}{{\raise1pt\hbox to0.35em{\Huge $\dot$}}}

\begin{document}
\date{October 15, 2024}

\title[A complete presentation of $(\mathbb{C}\mbox{[X]},\circ)$]
{A complete presentation of $(\mathbb{C}\mbox{[X]},\circ)$}

\author[B.R.~Berceanu]{Barbu Rudolf Berceanu}
\email{barberceanu@yahoo.com}

\thanks{$^2$ This research was partially supported by Viste arquitectura}



\begin{abstract}
We recall the fundamental theorem of J.F. Ritt, with a stress on the action of the affine 
group and canonical forms of complex polynomials. Then we give a complete presentation of 
the monoid $(\mathbb{C}\mbox{[X]},\circ)$. A list of decomposable polynomials is
given in degrees $\leq 12$.
\end{abstract}

\maketitle
\setcounter{tocdepth}{1}
\tableofcontents

\section{Introduction and statement of results}
\label{sec:intro}

Let $K$ be a field of characteristic 0 which is \textit{radically closed}: $K$ is the 
splitting field of polynomials $X^n-a$ (for any $a\in K$ and any $n\geq 2$).

The group of units in the monoid $(K[X],\circ)$ is 
$$ {\rm Aff}(K)=\{\alpha\in K[X]\mid\alpha(X)=aX+b,a\neq 0)\}. $$
Two polynomials $P(X),Q(X)$ in $K[X]$ are \textit{associate}, $P(X)\sim Q(X)$, if there are 
two invertible polynomials $\alpha,\beta\in {\rm Aff}(K)$ such that 
$$ \alpha(X)\circ P(X)\circ \beta^{-1}(X)=Q(X), $$
or, equivalently, if $P(X)$ and $Q(X)$ are $G$-equivalent under the natural action of the group 
$G={\rm Aff}(K)\times{\rm Aff}(K)$ on $K[X]$:
$$ (\alpha,\beta)\cdot P=\alpha\circ P\circ\beta^{-1}. $$
The $G$-orbit of $P(X)$ will be denoted $\widehat{P(X)}$. To find simple polynomials in a 
given $G$-orbit we introduce the following definition:
\begin{definition}\label{def11} 
A polynomial $P(X)$ of degree $d$ is a $\delta$-\textit{polynomial} if 
$$  \begin{array}{l}
     P(X)=X^d, \quad \mbox { or}     \\
     P(X)=X^d+dX^e+a_{e-1}X^{e-1}+\ldots+a_1X,  \mbox{ where }1\leq e\leq d-2.
\end{array} $$
\end{definition}
Let $L$ be an extension of $K$.
\begin{definition}\label{def12} 
A polynomial $P(X)$ in $K[X]$ is {\em decomposable} in $L[X]$ if there are two polynomials 
$Q,R\in L[X]$, ${\rm deg}(Q)=q$, ${\rm deg}(R)=r$ and $q,r\geq 2$, such that
$$ P(X)=Q(X)\circ R(X); $$
we say that $P(X)$ has a $(q,r)$ decomposition. Otherwise $P(X)$, a polynomial of degree 
$\geq 2$, is {\em indecomposable}. 
\end{definition}
In Section 2  we will introduce $\mu-$, $\nu-$ and $\varrho-$ polynomials, in order to obtain simple 
factors in a decomposition of a polynomial, $P(X)=Q(X)\circ R(X)$. In Proposition \ref{prop22} we show 
that every $G$-orbit contains finitely many $\delta$-polynomials. We describe an algorithm to decide if 
two polynomials are associate. In a $G$-orbit, not containing a monomial, the isotropy groups are finite,
see Proposition \ref{prop23}.

In Section 3, in Theorem \ref{thm33}, we show that one can find a 'normal' decomposition, see Definition
\ref{def32}, and its two components are in $K[X]$. As a consequence, an algorithm is given to decide
if a polynomial $P(X)$ is indecomposable or to find a decomposition of $P(X)$.

In Section 4 we give examples of indecomposable polynomials; for instance, we show that the cyclotomic
polynomial $\Phi_n(X)$ ($n\geq 3$) is indecomposable only if $n$ is square free or $n=4$, see 
Proposition \ref{prop41}. We also give conditions to have indecomposable Ritt polynomials, see 
Definition \ref{def48}.

In Section 5 we recall two fundamental results of J.~F.~Ritt, Theorem \ref{thm54} and Theorem \ref{thm56}.
The fundamental generators of the monoid $(\mathbb{C}\mbox{[X]},\circ)$, those contained in Ritt 
presentation, are ${\bf M}_p=X^p$, ${\bf T}_p(X)=\cos(p\arccos X)$ and Ritt polynomials 
${\bf R}_\lambda^{p,s}(X)$, ${\bf R}_\rho(X)^{p,s}$, where $p$ is a prime number, $1\leq s\leq p-1$.

In Section 6 we solve the 'ambiguities' of Ritt presentation and we find a complete presentation of
the monoid $(\mathbb{C}\mbox{[X]},\circ)$:
\begin{theorem}\label{thm13}
A complete presentation of the monoid $(\mathbb{C}[X],\circ)$ is given by 
$$ \left(
\begin{array}{cll} 
                      & \mid & {\bf M}_p\circ{\bf M}_q\succ{\bf M}_q\circ{\bf M}_p,\,(p>q
                                                                \mbox{ are prime numbers})             \\
\mbox{all}            & \mid & {\bf T}_p\circ{\bf T}_q\succ{\bf T}_q\circ{\bf T}_p,\,(p>q\geq 3
                                                                      \mbox{ are prime numbers})       \\
\mbox{indecomposable} & \mid & {\bf R}_\lambda^{p,s}\circ{\bf M}_p\succ{\bf M}_p\circ{\bf R}_\rho^{p,s}
                                                                      \,(p\mbox{ is prime },1\leq s<p) \\
\mbox{polynomials}    & \mid & {\bf R}_\lambda^{p,s}\circ{\bf M}_{q_1}\circ\ldots\circ{\bf M}_{q_s}
                              \circ{\bf  M}_p\succ{\bf M}_p\circ{\bf R}_\rho^{p,s}\circ{\bf M}_{q_1}\circ
                                                                           \ldots\circ{\bf M}_{q_s}    \\
                      & \mid & (p>q_s\geq q_{s-1}\geq\ldots\geq q_1\mbox{ are prime numbers})
\end{array} 
\right) $$
(the relations are up to a $G$-equivalence).
\end{theorem}

The Appendix contains all decomposable $\delta$-polynomials of degree $\leq 12$.


\section{Canonical forms of polynomials}

We will see that, for nice factors in the decomposition of a polynomial, we need the following 
(not so simple) polynomials:
\begin{definition}\label{def21} 
A non constant polynomial $P(X)$ in $K[X]$ will be called:

a) an $\mu-$\textit{polynomial} if it is monic;

b) an $\nu-$\textit{polynomial} if 0 is one of its roots;

c) an $\varrho-$\textit{polynomial} if it is reduced, that is if
$$  P(X)=\begin{cases}
  a_dX^d, \mbox{ where }a_d\neq 0,\quad \mbox { or}                                  \\
  a_dX^d+a_eX^e+\ldots+a_1X,  \mbox { where }a_d,a_e\neq 0,
     \mbox { and }1\leq e\leq d-2.
\end{cases} $$
If $P(X)$ is an $\varrho-$polynomial ($P(X)\neq a_dX^d$), $e$ is its \textit{second degree},  
$(d,e)$ is its \textit{bi-degree} and the difference $\Delta=d-e$ is the \textit{descent} of $P(X)$. 
The \textit{support} of an $\varrho-$polynomial $P(X)$ is the decreasing sequence 
$(d,e,\ldots,e_i,\ldots)$ corresponding to monomials $X^{e_i}$ in $P(X)$ with non-zero coefficients. 
In the case $P(X)=a_dX^d$ ($a_d\neq 0$), the descent is $d$ and the support is $(d)$. By convention,
the support of a constant polynomial $P(X)=a_0$ is $(0)$ (even in the case $a_0=0$). 
\end{definition}
\begin{example}
$2X^5-3X^2+X$ is an $\varrho-$polynomial with bi-degree $(5,2)$ and support $(5,2,1)$; $X^2+1$ is an 
$\mu-$polynomial, but it is not an $\varrho-$polynomial.
\end{example} 

\begin{prop}\label{prop22}
a) In one $G$-orbit $\mathcal{O}$, all polynomials have the same degree and all $\varrho-$polynomials 
have the same bi-degree and the same support.

b) Every $G$-orbit $\mathcal{O}$ contains a $\delta-$polynomial $Q(X)$. If a $\delta-$polynomial 
$Q(X)$ in $\mathcal{O}$ has bi-degree $(d,e)$, then $\mathcal{O}$ contains at most $\Delta=d-e$ 
$\delta-$polynomials.    
\end{prop}
\begin{proof}
a) Take $(\alpha,\beta)\in G$, $\alpha(X)=aX+b$, $\beta^{-1}(X)=rX+s$ and 
$$ P(X)=c_dX^d+c_{d-1}X^{d-1}+\ldots+c_0, \quad a,r,c_d\neq 0. $$
It is obvious that $\alpha\circ P$ and $P\circ \beta^{-1}$ have degree $d$. 

If $P(X)\sim Q(X)$ are $\varrho-$polynomials of degree $d$, then the coefficient of $X^{d-1}$ in
$Q=\alpha\circ P\circ\beta^{-1}$ is $ac_ddr^{d-1}s$, hence $s=0$ (in the case $d=0$, $Q$ is a
constant, too, possibly $0$). The coefficients of $X^i$ ($i\geq 1$) in $P$ and $Q$ are $c_i$ and 
$ac_ir^i$, therefore $P$ and $Q$ have the same the same support and bi-degree.

b) Only in this part we need a radically closed field $K$. Choose a polynomial $P(X)$ in $\mathcal{O}$. 
Using a Tschirnhaus transformation
$$ P(X)\mapsto P(X)\circ \left(X-\dfrac{c_{d-1}}{dc_d}\right)=R(X) $$
and $\alpha(X)=c_d^{-1}X-c_d^{-1}R(0)$ we obtain an $\mu\varrho-$polynomial
$$  \begin{array}{lll}
  S(X) & = & \left(\dfrac{1}{c_d}X-\dfrac{1}{c_d}R(0)\right)\circ P(X)\circ\left(X-\dfrac{c_{d-1}}{dc_d}\right) \\
       & = & \begin{cases}
       X^d+f_eX^e+f_{e-1}X^{e-1}+\ldots+f_1X\,(\mbox{where }f_e\neq 0,\,1\leq e\leq d-2),     \\
       X^d.
\end{cases}  
\end{array} $$
If $S(X)=X^d$, this is the unique $\delta$-polynomial in its orbit. If $f_e\neq 0$ and $1\leq e\leq d-2$, 
take $v$ a root of $X^{d-e}-\frac{d}{f_e}$, $\eta(X)=v^dX$, $\theta(X)=vX$, and we obtain the 
$\delta-$polynomial
$$ Q(X)=\eta(X)\circ P(X)\circ \theta^{-1}(X)=X^d+dX^e+g_{e-1}X^{e-1}+\ldots+g_1X. $$
The other $\delta-$polynomials associate with $Q(X)$ can be obtained only by acting with 
$(\lambda(X)=u^dX,\mu(X)=uX)\in G$, where $u$ is a root of unity of order $\Delta=d-e$.
\end{proof}
\begin{example}
Take the  $\delta-$polynomials $Q_{\alpha}=X^9+9X^5+\alpha X^4$, $R_{\theta}=X^9+9X^5+\theta X^3$  and 
$S_{\sigma}=X^9+9X^5+\sigma X$ in $\mathbb{C}\mbox{[X]}$. Then 
$$ \begin{array}{lll}
Q_{\alpha}\sim Q_{\beta} & \mbox{ if and only if } & \alpha=\pm\beta\mbox{ or }\alpha=\pm i\beta,  \\
R_{\theta}\sim R_{\eta}  & \mbox{ if and only if } & \theta=\pm\eta,                               \\
S_{\sigma}\sim S_{\tau}  & \mbox{ if and only if } & \sigma=\tau.        
\end{array} $$
\end{example} 
\begin{remark}
The proof of Proposition \ref{prop22} gives an algorithm to decide if two polynomials $P_1(X)$ and $P_2(X)$ 
are associate: we compute two $\delta-$polynomials $Q_1(X)$ and $Q_2(X)$ corresponding to $P_1(X)$ and 
$P_2(X)$. If $Q_1$ and $Q_2$ have the same support $(d,\delta,e_1,e_2,\ldots,e_s)$, we compute the ratios 
$u_j=a_{e_j}b^{-1}_{e_j}$ (here $a_{\ast}$ and $b_{\ast}$ are the coefficients of $Q_1$ and $Q_2$ 
respectively). Then $P_1$ and $P_2$ are associate if and only if there is a root of unity $u$ of order 
$\Delta=d-e$ such that, for all $j\in\{1,2,\ldots,s\}$, we have $u_j=u^{d-e_j}$.
\end{remark}
\begin{example}
There are unique $G$-orbits for polynomials of degree 0, 1 and 2:
$$ \mathcal{O}^0=\widehat{1},\,\mathcal{O}^1=\widehat{X} \mbox{ and }\mathcal{O}^2=\widehat{X^2}. $$
In degree 3 there are only two $G$-orbits:
$$ \mathcal{O}^3=\widehat{X^3} \mbox{ and }\mathcal{O}^{3,1}=\widehat{X^3+3X}. $$
For higher degrees there are infinitely many $G$-orbits. For instance, in degree 4 these are:
$$ \mathcal{O}^4=\widehat{X^4},\,\mathcal{O}^{4,1}=\widehat{X^4+4X}\mbox{ and }
\mathcal{O}^{4,2}_{\alpha}=\widehat{X^4+4X^2+\alpha X}. $$ 
We have $\mathcal{O}^{4,2}_{\alpha}=\mathcal{O}^{4,2}_{\beta}$ if and only if $\alpha=\pm\beta$.
\end{example} 
\begin{prop}\label{prop23}
The isotropy groups of the action of $G$ on $K[X]$ are given by
\begin{center}
\begin{tabular}{|c|c|c|c|c|}
\hline
$P(X)$     & $1$ & $X$            & $X^d\,(d\geq 2)$ & $ X^d+dX^e+\ldots+a_{e_i}X^{e_i}+\ldots $ \\
\hline
$G_{P(X)}$ & $K^{\ast}\times{\rm Aff}(K)$ 
                & ${\rm Aff}(K)$ & $K^{\ast}$       & $ \mathbb{Z}_n \mbox{ where } n=
                                                         {\rm gcd}(\Delta,\ldots,d-e_i,\ldots) $ \\
\hline
\end{tabular}
\end{center}
\end{prop}
\begin{proof}
The isotropy groups of the polynomial $1$ is isomorphic with $K^{\ast}\times{\rm Aff}(K)$: if 
$\alpha\circ 1=1\circ \beta$, then $\alpha(X)=aX+(1-a)$ and there is no restriction on $\beta(X)$. 
For higher degrees it is enough to compute isotropy groups 
$G_{P(X)}$ for $P(X)$ a $\delta-$polynomial. In degree 1, from $(aX+b)\circ X=X\circ(cX+d)$, we obtain the group 
$$ \{(aX+b,aX+b)\mid a\neq 0\}\cong {\rm Aff}(K). $$ 
If $d\geq 2$, from $(aX+b)\circ X^d=X^d\circ(cX+e)$ we get $e=b=0$ and the group
$$ \{(a^dX,aX)\mid a\neq 0\}\cong K^{\ast}. $$ 
In the proof of Proposition \ref{prop22} we saw that $X^d+dX^e$
is invariant under the transformations $(u^dX,uX)$, where $u$ is a root of unity of order $\Delta$; these leave 
invariant the monomial $a_{e_j}X^{e_j}$ ($a_{e_j}\neq 0$) if and only if $u^{d-e_j}=1$, hence $u$ is a root 
of unity of order $n={\rm gcd}(\Delta,\ldots,d-e_j,\ldots) $ and the isotropy group is
$$ \{(u^dX,uX)\mid u^n=1\}\cong \mathbb{Z}_n . $$  
\end{proof}


\section{Decomposition of polynomials}

We begin to show that decomposability is compatible with the $G$-action and it is independent of 
the field extension.
\begin{prop}\label{prop31}
If $P(X)$ and $Q(X)$ are associate, then $P(X)$ have an $(a,b)$ decomposition if and only if $Q(X)$ 
have an $(a,b)$ decomposition.
\end{prop}
\begin{proof}
If $\alpha\circ P\circ \beta^{-1}=Q$ and $P=A\circ B$, then $Q=(\alpha\circ A)\circ (B\circ \beta^{-1})$ 
and the degrees of $A(X)$ and $B(X)$ are those of $\alpha(X)\circ A(X)$ and $B(X)\circ \beta^{-1}(X)$
respectively.
\end{proof}
\begin{definition}\label{def32} 
A $(q,r)$ decomposition of $P(X)$, $P=Q\circ R$, ($P(X)\in K[X]$, $Q(X),R(X)\in L[X]$) is a 
\textit{normal} $(q,r)$ \textit{decomposition} if $R(X)$ is a $\mu\nu-$polynomial.
\end{definition}
\begin{theorem}\label{thm33}
a) Suppose that $P(X)$, a polynomial of degree $n$ in $K[X]$, has a $(q,r)$ decomposition in $L[X]$.
Then $P(X)$ has a unique normal $(q,r)$ decomposition $P(X)=Q(X)\circ R(X)$. Moreover, $Q(X)$ and 
$R(X)$ are polynomials in $K[X]$.

b) If  $Q(X)\circ R(X)$ is the normal $(q,r)$ decomposition of $P(X)\in K[X]$, then

$\quad\quad b_1)$ $P(X)$ is a $\mu-$ (or a $\nu-$)polynomial if and only if $Q(X)$ is a $\mu-$ 
(or a $\nu-$)polynomial;

$\quad\quad b_2)$ $P(X)$ is a $\varrho-$polynomial if and only if $R(X)$ is a $\varrho-$polynomial. In
this case, if $R(X)\neq X^r$, then $P(X)$ and $R(X)$ have the same descent $\Delta$.   
\end{theorem}
\begin{proof}
a) Choose a $(q,r)$ decomposition $P(X)=S(X)\circ T(X)$ in $L[X]$, where 
$$ T(X)=e_rX^r+e_{r-1}X^{r-1}+\ldots +e_0\in L[X],\quad (e_r\neq 0,2\leq r\leq n-2). $$
Using $\alpha(X)=e_r^{-1}X-e_r^{-1}e_0$, we find a normal $(q,r)$ decomposition 
$$ P(X)=[S(X)\circ \alpha^{-1}(X)]\circ [\alpha(X)\circ T(X)]=Q(X)\circ R(X) $$
(here ${\rm deg}(Q)={\rm deg}(S)$, ${\rm deg}(R)={\rm deg}(T)$, and $R(X)$ is an $\mu\nu-$polynomial).

Looking at the first coefficients in the equality $P(X)=Q(X)\circ R(X)$:
$$\begin{array}{ll}
P(X)=a_nX^n+a_{n-1}X^{n-1}+\ldots +a_{n-r+1}X^{n-r+1}+\ldots +a_0, &  (a_n\neq 0,a_i\in K),  \\
Q(X)=b_qX^q+b_{q-1}X^{q-1}+\ldots +b_0,                          &  (b_q\neq 0,b_j\in L),    \\
R(X)=X^r+c_{r-1}X^{r-1}+\ldots +c_1X,                        &  (c_k\in L,n=qr),
\end{array}$$
we find that 
$$ \begin{array}{lll}
a_n           & =      & b_q                                                       \\
a_{n-1}       & =      & b_q{q\choose 1}c_{r-1}                                    \\
a_{n-2}       & =      & b_q[{q\choose 1}c_{r-2}+{q\choose 2}c_{r-1}^2]=
                               b_q[{q\choose 1}c_{r-2}+p_{n-2}(c_{r-1})]           \\ 
\ldots \ldots & \ldots & \ldots \ldots \ldots \ldots \ldots  \ldots \ldots \ldots  \\
a_{n-r+1}    & =      & b_q[{q\choose 1}c_1+p_{n-r+1}(c_{r-1},\ldots ,c_2)]
\end{array} $$
(where $p_{\ast}(..,c_j,..)$ are polynomials in $\{c_j\}$ with coefficients in
$\mathbb{Z}$), therefore $b_q\in K$ and $R(X)$, uniquely determined by $P(X)$, have the coefficients in $K$. 
For later use, let us remark that $p_{\ast}(..,c_j,..)$ is 0 if all $c_j$ are 0. The 
coefficients of $X^{n-jr}$ ($j=1,\ldots,q$) will give
$$\begin{array}{lll}
a_{n-r}      & =    & b_{q-1}+b_qp_{n-r,q}(c_{r-1},\ldots ,c_1)                                            \\
a_{n-2r}     & =    & b_{q-2}+b_{q-1}p_{n-2r,q-1}(c_{r-1},\ldots ,c_1)+b_qp_{n-2r,q}(c_{r-1},\ldots ,c_1)  \\
\ldots\ldots &\ldots& \ldots \ldots \ldots \ldots \ldots  \ldots \ldots \ldots  \ldots \ldots \ldots\ldots \\
a_r          & = & b_1+\sum_{j=2}^qb_jp_{r,j}(c_{r-1},\ldots ,c_1)
\end{array}$$
and finally $a_0=b_0$. This shows that $Q(X)$ is uniquely determined by $P(X)$ and by the pair $(q,r)$,
and also that $Q(X)\in K[X]$.

b) $P(X)$ and $Q(X)$ are at the same time $\mu-$polynomials because $a_n=b_q$. They are simultaneously
$\nu-$polynomials because $a_0=b_0$. The last equivalence comes from the relation $a_{n-1}=b_qqc_{r-1}$.
If the descent of $R(X)(\neq X^r)$ is $r-e$, then
$$ a_{n-1}=qc_{r-1}=0, a_{n-2}=qc_{r-2}=0,\ldots ,a_{n-r+e+1}=qc_{e+1}=0,a_{n-r+e}=qc_e\neq 0, $$
hence $\Delta_P=n-(n-r+e)=r-e=\Delta_R$. If $R(X)=X^r$, then $\Delta_P=\Delta_R$ only when $Q(X)$ is
not a $\varrho-$polynomial.
\end{proof}
In the rest of the paper we suppose that 
polynomials are $\delta-$polynomials, whenever this is possible.
\begin{corollary}\label{cor34}
($\delta$ decomposition) If a $\delta-$polynomial $P(X)$ has the normal $(q,r)$ decomposition 
$P(X)=Q(X)\circ R(X)$, then $R(X)$ is also a $\delta-$polynomial. In the special case $R(X)=X^r$,
$Q(X)$ is a $\varrho-$polynomial if and only if $r\neq\Delta_P$; in this case $\Delta_P=r\Delta_Q$.
\end{corollary}
\begin{proof}
If the polynomial $R(X)$ is not $X^r$, from the proof of Theorem  \ref{thm33}, we find that
$qr=n=a_e=qc_{r-n+e}$, hence the second non-zero coefficient of $R(X)$ is $r$. If $R(X)=X^r$ and
$ Q(X)=X^q+b_fX^f+\ldots +b_1X$, where $b_f\neq 0$, then $\Delta_P=r\Delta_Q$. In this case 
$b_f=n\neq q$, hence $Q(X)$ is not a $\delta-$polynomial. Only in the case $Q(X)=X^q$, $Q$ 
is a $\delta-$polynomial and again $\Delta_P=n=qr=r\Delta_Q$.
\end{proof}
\begin{remark}
If $P(X)=Q(X)\circ R(X)$ where $Q(X)$ is an $\mu\nu-$polynomial and $R(X)$ is a $\delta-$polynomial, then 
$P(X)$ is an $\mu\varrho-$polynomial (if $R(X)\neq X^r$, then $P(X)$ is a $\delta-$polynomial).
\end{remark}
\begin{example}
The $\delta$ decomposition of a $\delta-$polynomial looks nicer than a decomposition with the 
first component a reduced polynomial:
$$ X^6+6X^4+8X^2=(X^3+6X^2+8X)\circ X^2=(X^3-4X)\circ (X^2+2). $$  
\end{example}
\begin{remark}
The proof of Theorem \ref{thm33} gives an algorithm to decide if a polynomial $P(X)$ is decomposable 
(and, in this case, to find a decomposition): 
if $n={\rm deg}(P)$ is prime, then $P(X)$ is indecomposable. For a pair $(q,r)$ such that $n=qr$, 
we check if there is a $(q,r)$ decomposition of $P(X)$: we can work with $P_1(X)$, a $\delta-$polynomial
associated with $P(X)$. Using the formulae in the proof of Theorem \ref{thm33}, we compute the
coefficients of the two associated polynomials $Q(X)$ and $R(X)$ in a normal $(q,r)$ decomposition. 
Finally we compute the composition $Q(X)\circ R(X)=P_2(X)$: if $P_1=P_2$, then $P(X)$ have a $(q,r)$ 
decomposition; otherwise we restart with another pair $(q',r')$ satisfying $n=q'r'$, if any.
\end{remark}

An $(q_1,q_2,\ldots,q_s)$ \textit{decomposition } of the polynomial $P(X)\in K[X]$ is given by
$$ P(X)=Q_1(X)\circ Q_2(X)\circ\ldots\circ Q_s(X),  $$
where $Q_j(X)$ are polynomials with coefficients in some extension $K<L$ and ${\rm deg}(Q_j)=q_j\geq 2$.
From Proposition \ref{prop31} and Theorem \ref{thm33} this decomposition holds for the polynomials 
associated with $P(X)$ and we can find such a decomposition with $Q_j(X)\in K[X]$. Another consequence 
of these and of Corollary \ref{cor34} is the following result.
\begin{corollary}\label{cor35}
($\mu\nu\mu\nu\delta$ decomposition) If a $\delta-$polynomial $P(X)$ in $K[X]$ has a 
$(q_1,q_2,\ldots,q_s)$ decomposition, then there is a unique decomposition 
$$ P(X)=Q_1(X)\circ Q_2(X)\circ\ldots\circ Q_s(X),\, {\rm deg}(Q_j)=q_j,  $$
where $Q_1(X),\ldots,Q_{s-1}(X)$ are $\mu\nu-$polynomials and $Q_s(X)$ is a $\delta-$polynomial.
\end{corollary}
Let us denote by $\mathcal{D}^{q_1,\ldots,q_s;\Delta}$ the set of $\delta-$polynomials
$$ \{P(X)\in K[X]\mbox{ with descent }\Delta\mbox{ and a }\mu\nu\mu\nu\delta\,(q_1,\ldots,q_s)
            \mbox{ decomposition}\}. $$
We consider this set, if non-empty, as a vector subspace of the vector space of coefficients of 
monic polynomials $Q_1,Q_2,\ldots,Q_s$:
$$ \mathcal{D}^{q_1,\ldots,q_s;\Delta}<K^{q_1}\times K^{q_2}\times\ldots \times K^{q_s}. $$ 
\begin{prop}\label{prop36}
We have
$$ {\rm dim}(\mathcal{D}^{q_1,\ldots,q_s;\Delta}) =\begin{cases}
 \sum_{j=1}^sq_j-\Delta-s,    & \mbox{ if }q_s>\Delta,                                                \\
 \sum_{j=1}^rq_j-\Delta_1-r   & \mbox{ if }\Delta=\Delta_1\prod_{j=r+1}^sq_j,\,\Delta_1\in\mathbb{N}
                                                                          \mbox{ and } q_r>\Delta_1,  \\
 0,                           & \mbox{ if }\Delta=\prod_{j=1}^sq_j.                                   \\
   \end{cases} $$
Otherwise the set $\mathcal{D}^{q_1,\ldots,q_s;\Delta}$ is empty.   
\end{prop}
\begin{proof}
If $Q_s\neq X^{q_s}$, then 
$$ \begin{array}{l}
Q_j(X)=X^{q_j}+a_{j,q_j-1}X^{q_j-1}+\ldots +a_{j,1}X,\,(j=1,\ldots,s-1), \\
Q_s(X)=X^{q_s}+q_sX^e+a_{s,e-1}X^{e-1}+\ldots +a_{s,1}X,\,\Delta=q_s-e\leq q_s-1\mbox{ and } \\
\quad\quad {\rm dim}(\mathcal{D}^{q_1,\ldots,q_s;\Delta})=\sum_{j=1}^{s-1}(q_j-1)+(e-1)=
             \sum_{j=1}^s q_j-\Delta-s.     
\end{array} $$
If $Q_r(X)=X^{q_r}+a_{r,q_r-\Delta_1}X^{q_r-\Delta_1}+\ldots +a_{r,1}X$, where 
$a_{r,q_r-\Delta_1}\neq 0$, and $Q_j(X)=X^{q_j}$, $j=r+1,\ldots,s$, then
$$ \Delta=\Delta_1q_{r+1}\ldots q_s, \,a_{r,q_r-\Delta_1}=q_rq_{r+1}\ldots q_s $$
and
$ {\rm dim}(\mathcal{D}^{q_1,\ldots,q_s;\Delta})=\sum_{j=1}^{r-1}(q_j-1)+(q_r-\Delta_1-1)=
             \sum_{j=1}^r q_j-\Delta_1-r. $\\
If $Q_j(X)=X^{q_j}$ for any $j\in\{1,2,\ldots,s\}$, then $\Delta=\prod_{j=1}^s q_j$ and 
$$  \mathcal{D}^{q_1,\ldots,q_s;\Delta}=\{X^n=X^{q_1}\circ X^{q_2}\circ\ldots\circ X^{q_s}\}. $$              
\end{proof}


\section{Indecomposable polynomials}\label{sectin}

We begin with a few examples of indecomposable polynomials, most of them will appear in a 
presentation of the monoid $(\mathbb{C}\mbox{[X]},\circ)$.

{\bf Prime degrees.} If the degree of $P(X)$ is a prime number, then it is indecomposable.

{\bf Monomials.} $X^n$ in indecomposable if and only if $n$ is a prime number.

{\bf Chebyshev polynomials.} ${\bf T}_n(X)=\cos(n\arccos X)$ is indecomposable if and only if $n$ 
is a prime number: the degree of ${\bf T}_n(X)$ is $n$ and we have 
${\bf T}_{mk}(X)={\bf T}_m(X)\circ {\bf T}_k(X)$.

Using the recurrence relation ${\bf T}_{n+1}(X)=2X{\bf T}_n(X)-{\bf T}_{n-1}(X)$, we find that
 ${\bf T}_n(-X)=(-1)^n{\bf T}_n(X)$ and also that
$$\begin{array}{lll}
{\bf T}_2(X) & = & 2X^2-1   \sim  X^2,                                            \\
{\bf T}_3(X) & = & 4X^3-3X  \sim  X^3+3X,                                         \\
{\bf T}_n(X) & = & 2^{n-1}X^n-n2^{n-3}X^{n-2}+n(n-3)2^{n-6}X^{n-4}+\ldots \sim    \\
             & \sim & X^n+nX^{n-2}+\frac{n(n-3)}{2}X^{n-4}+\ldots \, (n\geq 4). 
\end{array}    $$
As a consequence, for $n$ odd, $T_n(X)=2^{1-n}{\bf T}_n(X)$ in an $\mu\varrho$-polynomial with descent 2 
and also that $T_n(X)=XU_n(X^2)$. 

{\bf Cyclotomic polynomials.} $\Phi_n(X)$ is the cyclotomic polynomial, the minimal monic polynomial of a 
primitive root of unity of order $n$, with rational coefficients.
\begin{prop}\label{prop41}
$\Phi_n(X)$ is indecomposable if and only if $n$ is a product of distinct 
prime numbers or $n=4$. 
\end{prop}
One implication is given by the following proposition (see \cite{L}, chapter VI,3, for a similar relation, 
under the correct hypothesis $p^2\mid n$); the proof of the second implication will be given in Proposition 
\ref{prop44}. The polynomial $\Phi_{mk}(X)$ has degree 1 if and only if $mk\leq 2$ and this explains the 
special case $n=4$.
\begin{prop}\label{prop42}
$$ \Phi_{mk^2}(X)=\Phi_{mk}(X)\circ X^k. $$
\end{prop}
\begin{proof}
The numbers $\xi=e^{\frac{2\pi i}{mk^2}}$, $\eta=\xi^k$ and $\omega=\xi^{mk}$ are primitive roots of unity 
of orders $mk^2$, $mk$ and $k$ respectively. In the factorization 
$$ \Phi_{mk}(X)\circ X^k=\prod_{(a,mk)=1}(X^k-\eta^a)=\prod_{(a,mk)=1}\prod_{b=0}^{k-1}(X-\xi^a\omega^b)=
  \prod_{(a,mk)=1}^{b=0,..,k-1}(X-\xi^{a+bmk}) $$ 
the roots are distinct primitive roots of unity of order $mk^2$ (if $a$ and $mk$ are coprime, so are 
$a+bmk$ and $mk^2$). Euler function satisfies the relation $\varphi(mk^2)=\varphi(mk)k$ and this shows that
the monic separable polynomials $\Phi_{mk^2}(X)$ and $\Phi_{mk}(X)\circ X^k$ have the same degree and the 
same roots, hence they coincide.
\end{proof}

\begin{prop}\label{prop43}
In $\Phi_n(X)=X^{\varphi(n)}+c_{\varphi(n)-1}X^{\varphi(n)-1}+\ldots+c_1X+c_0$, $n\geq 2$, we have 
$$ \{ c_{\varphi(n)-1},c_1,c_0\}\subseteq \{-1,0,1\}. $$
More precisely, 
$$ \begin{array}{lll} 
c_1=c_{\varphi(n)-1} & = & \begin{cases}
      (-1)^{s+1}, & \mbox{if }n=p_1p_2\ldots p_s,(\mbox{distinct prime numbers})    \\
      0, & \mbox{if there is a prime }p \mbox{ such that }p^2\mid n.                \\
                           \end{cases}              \\
   c_0 & = & 1.
\end{array}    $$
\end{prop}
\begin{proof}
In the case $p^2\mid n$, $c_1=0$ is a consequence of Proposition \ref{prop42}. If $n$ is a product of 
$s$ distinct prime numbers, the proof is by induction. If $s=1$ we have 
$$ \Phi_p(X)=X^{p-1}+X^{p-2}+\ldots+X+1. $$
The sum $S$ of all the roots of unity of order $n=p_1p_2\ldots p_s$ is zero and it is also equal 
with the following sum, where $SP_m$ is the sum of primitive roots of unity of order $m$:
$$ \begin{array}{lll}
S & = & 1+\sum_{i=1}^sSP_{p_i}+\sum_{i<j}SP_{p_ip_j}+\ldots+\sum_{k=1}^sSP_{p_1..\hat{p_k}..p_s}+SP_n= \\
  & = & 1+{s\choose 1}(-1)+{s\choose 2}(-1)^2+\ldots+{s\choose s-1}(-1)^{s-1}+SP_n,
\end{array} $$
hence the last term is $(-1)^s$. Finally, the polynomial $\Phi_n(X)$ is reciprocal.
\end{proof}

\begin{prop}\label{prop44}
$\Phi_{p_1p_2\ldots p_s}(X)$ is indecomposable.
\end{prop}
\begin{proof}
Suppose that $\Phi_{p_1\ldots p_s}(X)=Q(X)\circ R(X)$ is a normal $(q,r)$ decomposition. From 
Proposition \ref{prop43} and Theorem \ref{thm33} we find that 
$$ R(X)=X^r+\frac{(-1)^{s+1}}{q}X^{r-1}+\ldots+c_1X, $$
where $q=\deg(Q)$ is an integer greater than 1. If $\alpha$ is one root of $Q(X)$, then all the roots 
of the polynomial $R(X)-\alpha$ are (primitive) roots of unity and their sum (an algebraic integer) 
equals $\frac{(-1)^s}{q}$, which is not an algebraic integer. 
\end{proof}
{\bf Ritt polynomials.} We denote by $R_\lambda$ and $R_\rho$ the polynomials: 
$$ R_\lambda(X)=X^sG^n(X),\, R_\rho(X)=X^sG(X^n), $$
where $G(X)$ is an $\mu$-polynomial in $K[X]$ with $\deg G\geq 2$ and $G(X)\neq X^m$.
\begin{lemma}\label{lema45}
a) If ${\rm gcd}(s,n)\geq 2$, then $R_\lambda(X)$ and $R_\rho(X)$ are decomposable.

b) If $kh=nj+s$ and $G(X)=X^jB^k(X)C[X^hB^n(X)]$, then $R_\lambda(X)=X^sG^n(X)$ and 
$R_\rho(X)=X^sG(X^n)$ are decomposable.
\end{lemma}
\begin{proof}
a) If $s=du$ and $n=dv$, $d\geq 2$, then 
$$ R_{\lambda}(X)=X^d\circ (X^uG^v(X)) \mbox{ and }R_{\rho}(X)=(X^uG(X^v))\circ X^d. $$

b) In this case we find 
$$ R_{\lambda}(X)=[X^kC^n(X)]\circ [X^hB^n(X)]\mbox{ and }R_{\rho}(X)=[X^kC(X^n)]\circ [X^hB(X^n)]. $$
\end{proof}
To obtain indecomposable $R_\ast$ polynomials, we will take relatively prime numbers $s,n$ and 
also we can choose $1\leq s\leq n-1$: 
$$ R_{\lambda}(X)=X^{s+jn}G_1^n(X)=X^sG^n(X),\, R_\rho(X)=X^{s+jn}G_1(X^n)=X^sG(X^n), $$
where $G(X)=X^jG_1(X)$, $G_1(0)\neq 0$. 

The following Lemma will be used in the proof of Proposition \ref{prop47}.
\begin{lemma}\label{lema46}
If $A,B,C$ are $\mu$-polynomials in $K[X]$ such that $A^n(X)=B(X)C(X)$, where $B(X)$ and $C(X)$ 
are relatively prime, then there are $\mu$-polynomials $B_1(X)$ and $C_1(X)$ in $K[X]$ such that
$$ B(X)=B_1^n(X),\,C(X)=C_1^n(X)\mbox{ and }A(X)=B_1(X)C_1(X).  $$
\end{lemma}
\begin{proof}
We can find monic polynomials $B_1,C_1$ with coefficients in $\overline{K}$, the algebraic 
closure of $K$, satisfying $B=B_1^n$ and $C=C_1^n$. From $B(X)=B_1^n(X)=X^n\circ B_1(X)$, 
$B(X)\in K[X]$, the proof of Theorem \ref{thm33} shows that $B_1(X)\in K[X]$.
\end{proof} 
\begin{prop}\label{prop47}
Let $p$ a prime number and $1\leq s<p$. If $R_\lambda(X)=X^sG^p(X)$ is decomposable, then 
there are $j,k,h\in\mathbb{N}$ and polynomials $B(X),C(X)$ in $ K[X]$ such that
$$ kh=s+jp\mbox{ and }G(X)=X^jB^k(X)C[X^hB^p(X)]. $$
The same conclusion holds if $R_\rho(X)=X^sG(X^p)$ is decomposable.
\end{prop}
\begin{proof}
From the normal $(q,r)$ decomposition $X^sG^p(X)=Q(X)\circ R(X)$, where 
$$\begin{array}{ll}
G(X)=X^m+a_{m-1}X^{m-1}+\ldots +a_jX^j=X^jG_1(X)  &  (a_j\neq 0),  \\
Q(X)=X^q+b_{q-1}X^{q-1}+\ldots +b_kX^k=X^kQ_1(X)  &  (b_k\neq 0),  \\
R(X)=X^r+c_{r-1}X^{r-1}+\ldots +c_hX^h=X^hR_1(X)  &  (c_k\neq 0),
\end{array}$$
we obtain 
$$ X^{s+jp}G_1^p(X)=X^{kh}R_1^k(X)Q_1[X^hR_1(X)], $$
hence $kh=s+jp$ because 0 is not a root of the polynomials $G_1(X)$, $R_1(X)$, $Q_1[X^hR_1(X)]$. 
This implies also that $R_1(X)$ and $Q_1[X^hR_1(X)]$ are relatively prime. Take a root $\xi$ 
of $R_1(X)$ with multiplicity $t$ ($\xi\in\overline{K}$, the algebraic closure of $K$). 
Let $u$ be the multiplicity of $\xi$ as a root of $G_1(X)$, therefore $up=tk$. 
Because $k,p$ are relatively prime, we find that  $p\mid t$. This implies that 
$R_1(X)=B^p(X)$. Let $\{\eta_1,\ldots,\eta_e\}$ be the distinct roots of $Q_1(X)$ with 
multiplicities $\{v_1,\ldots,v_e\}$. We will prove that all $v_\ast$ are multiple of $p$, 
hence $Q_1(X)=C^p(X)$ and $G_1(X)=B^k(X)C[X^hB^p(X)]$; this gives the required factorization 
of $G(X)$. Let $\{\theta_1,\ldots,\theta_f\}$ be the roots of $X^hB^p(X)-\eta_1$ with 
multiplicities $\{w_1,\ldots,w_f\}$. The root $\theta_1$ cannot be a root of another 
$X^hB^p(X)-\eta_\ast$ ($\eta_\ast\neq\eta_1$), hence the multiplicity of $\theta_1$ in 
$G_1^p(X)=B^p(X)Q_1[X^hB^p(X)]$ is $v_1w_1$ and it is also a multiple of $p$. If $v_1$ is not 
a multiple of $p$, then $p\mid w_1$: the same is true for all roots $\theta_\ast$ of 
$X^hB^p(X)-\eta_1$, hence $p\mid w_2,\ldots,p\mid w_f$ and
$$ {\rm deg}(X^hB^p(X)-\eta_1)=h+p\cdot{\rm deg}(B(X))=\sum w_\ast\equiv 0\,({\rm mod}\,p), $$
a contradiction ($kh=s+jp$ implies that $h,p$ are relatively prime). 

There are similar arguments if $R_\rho$ is decomposable. 
\end{proof}
\begin{definition}\label{def48} 
We define the \textit{Ritt polynomials} 
$$ {\bf R}_\lambda^{p,s}(X)=X^sG^p(X) \mbox{ and } {\bf R}_\rho^{p,s}(X)=X^sG(X^p), $$
where $p$ is a prime number, $1\leq s\leq p-1$ and $G(X)$ is a $\mu$-polynomial ($\deg G\geq 2$) 
which cannot be factorized into a product $G(X)=X^jB^k(X)C[X^hB^p(X)]$ with $kh=s+jp$ (in 
particular $G(X)\neq X^j$). It is obvious that ${\bf R}_\lambda^{p,s}(X)$ is an $\mu\nu$-polynomial 
and ${\bf R}_\rho^{p,s}(X)$ is an $\mu\varrho$-polynomial.
\end{definition}


\section{Ritt presentation}

\begin{lemma}\label{lema51}
a) If $P_1\circ Q=P_2\circ Q$, where $P_1,P_2,Q$ are polynomials of degrees $\geq 1$, 
then $P_1=P_2$.

b) If $P\circ Q_1=P\circ Q_2$, where $P,Q_1,Q_2$ are polynomials of degrees $\geq 1$, 
then there is a root of unity of order $\deg P$, $\omega$, and a constant $c$ such
that $Q_2=\omega Q_1+c$.
\end{lemma}
\begin{proof}
a) If
$$ \begin{array}{lll}
P_1(X) & = & a_nX^n+a_{n-1}X^{n-1}+\ldots +a_{k+1}X^{k+1}+a_kX^k+\ldots +a_1X+a_0, \\
P_2(X) & = & a_nX^n+a_{n-1}X^{n-1}+\ldots +a_{k+1}X^{k+1}+b_kX^k+\ldots +b_1X+b_0,
\end{array}   $$ 
the relation $P_1\circ Q=P_2\circ Q$ implies that $a_kQ^k(X)=b_kQ^k(X)$, hence $a_k=b_k$.

b) The relation $P\circ Q_1=P\circ Q_2$ and $\deg P\geq 1$ imply that $\deg Q_1=\deg Q_2$.
Consider the polynomials 
$$ \begin{array}{lll}
P(X)  & = & a_nX^n+a_{n-1}X^{n-1}+\ldots +a_0,\, (a_n\neq 0),    \\
Q_1(X) & = & b_qX^q+b_{q-1}X^{q-1}+\ldots +b_0,\, (b_q\neq 0),         \\
Q_2(X) & = & c_qX^q+c_{q-1}X^{q-1}+\ldots +c_0,\, (c_q\neq 0).      
\end{array}   $$
The leading terms in $P\circ Q_1$ and $P\circ Q_2$ imply that $a_nb_q^n=a_nc_q^n$, 
hence there is a root of unity of order $n$, $\omega$, such that $c_q=\omega b_q$.
Suppose that 
$$ Q_2(X)=\omega(b_qX^q+b_{q-1}X^{q-1}+\ldots +b_{k+1}X^{k+1}+d_kX^k+\ldots +d_0),\,
      (k\geq 1).    $$
The coefficient of $X^{(n-1)q+k}$ in $P\circ Q_1-P\circ Q_2$ is $na_nb_q^{n-1}(b_k-d_k)$,
because the degree of $(P-a_nX^n)\circ Q_1$ is at most $(n-1)q$. Therefore $Q_2-\omega Q_1$ is 
a constant polynomial.     
\end{proof}
\begin{example}
The polynomials $P(X)=X^2+2qX+r$, $Q_1(X)=X^2+aX+b$ and $Q_2(X)=-X^2-aX-b-2q$ satisfy
the relation $P\circ Q_1=P\circ Q_2$.
\end{example}
\begin{corollary}\label{cor52}
If the polynomials $P,Q,R,S$, of degrees $\geq 2$, verify the relation $P(X)\circ Q(X)=R(X)\circ S(X)$,
the following properties are equivalent:

i) there is a transformation $\alpha\in{\rm Aff}(K)$ such that $R(X)=P(X)\circ\alpha (X)$;

ii) there is a transformation $\beta\in{\rm Aff}(K)$ such that $S(X)=\beta (X)\circ Q(X)$.
\end{corollary}
\begin{proof}
Using Lemma \ref{lema51}, from $P\circ Q=P\circ\alpha\circ S$, we find a root of unity, $\omega$,
and a constant $c$ such that 
$$ Q(X)=\omega\cdot\alpha (X)\circ S(X)+c=
[\omega\cdot\alpha (X)+c]\circ S(X). $$
From $P\circ Q=R\circ\beta\circ Q$ we get
$P(X)=R(X)\circ\beta (X)$. 
\end{proof}

In the rest of the paper the polynomials are in $\mathbb{C}\mbox{[X]}$.
We recall the fundamental results of J.~F.~Ritt  \cite{R} in the theory of the monoid  
$(\mathbb{C}[X],\circ)$. The generators of this monoid are all irreducible polynomials. The 
following generators appear in the defining relations of this monoid:
$$ \begin{array}{l}
\mathcal{M}=\{{\bf M}_p=X^p \mid p\mbox{ a prime number}\},                            \\
\mathcal{T}=\{{\bf T}_p(X)=\cos(p\arccos X) \mid p\mbox{ a prime number}\geq 3\},      \\
\mathcal{R}=\{{\bf R}_\lambda^{p,s}(X),{\bf R}_\rho(X)^{p,s}\mid p,s, G\mbox{ as in 
                   Definition \ref{def48}}\}.
.                    
\end{array} $$
These will be called the \textit{fundamental generators}.
\begin{definition}\label{def53} 
An \textit{elementary relation} in $(\mathbb{C}[X],\circ)$ is an equality
$$ P(X)\circ Q(X)=R(X)\circ S(X)  $$
where $P,Q,R,S$ are indecomposable and $R(X)\notin P(X)\cdot{\rm Aff}(K)$ (equivalently,
$S(X)\notin {\rm Aff}(K)\cdot Q(X)$).

Two elementary relations $P\circ Q=R\circ S$ and $\overline P\circ\overline Q=\overline R\circ\overline S$
are $G$-\textit{equivalent} if there are $\alpha,\beta,\gamma,\delta\in {\rm Aff}(\mathbb{C})$ such that
$$ \begin{array}{l} 
\overline P=\alpha\circ P\circ\gamma^{-1},\overline Q=\gamma\circ Q\circ \beta^{-1}, 
    \overline R=\alpha\circ R\circ\delta^{-1}, \overline S=\delta\circ S\circ\beta^{-1} \mbox{ or} \\
\overline P=\alpha\circ R\circ\gamma^{-1},\overline Q=\gamma\circ S\circ \beta^{-1}, 
    \overline R=\alpha\circ P\circ\delta^{-1}, \overline S=\delta\circ Q\circ\beta^{-1}.    
\end{array}     $$ 
\end{definition}
\begin{example}
The elementary relations ${\bf T}_p(X)\circ{\bf T}_2(X)={\bf T}_2(X)\circ{\bf T}_p(X)$ ($p$ prime $\geq 3$) 
and ${\bf R}_\lambda^{2,1}(X)\circ{\bf M}_2={\bf M}_2\circ{\bf R}_\rho^{2,1}(X)$ (where $G(X)=U_p(X)$,
see Section \ref{sectin}) are $G$-equivalent.
\end{example}
\begin{theorem}\label{thm54}(J.~F.~Ritt, 1922)
Up to a $G$-equivalence, the list of elementary relations in the monoid  $(\mathbb{C}[X],\circ)$ is:
$$ \begin{array}{l} 
{\bf M}_p\circ{\bf M}_q={\bf M}_q\circ{\bf M}_p,\mbox{ where }p\neq q\mbox{ are prime numbers};        \\
{\bf T}_p\circ{\bf T}_q={\bf T}_q\circ{\bf T}_p,\mbox{ where }p\neq q\mbox{ are prime numbers }\geq 3; \\
{\bf R}_\lambda^{p,s}\circ{\bf M}_p={\bf M}_p\circ{\bf R}_\rho^{p,s}\mbox{ where }p\mbox{ is a prime number}.
\end{array} $$ 
\begin{definition}\label{def55}
Two decompositions of a polynomial $P(X)$ with indecomposable factors
$$ P_1(X)\circ\ldots\circ P_s(X)=Q_1(X)\circ\ldots\circ Q_s(X) $$
are \textit{elementary related} if there is $i\in\{1,\ldots,s-1\}$ such that $P_j(X)=Q_j(X)$ for 
$j\neq i,i+1$ and $P_i(X)\circ P_{i+1}(X)=Q_i(X)\circ Q_{i+1}(X)$ is an elementary relation.
\end{definition}
\end{theorem}

\begin{theorem}\label{thm56}(J.~F.~Ritt, 1922)
If $P_1(X)\circ\ldots\circ P_s(X)=Q_1(X)\circ\ldots\circ Q_t(X)$ are two decompositions of $F(X)$
with indecomposable factors, then $s=t$ and there are decompositions (with indecomposable factors)
$F(X)=R^{(i)}_1(X)\circ\ldots\circ R^{(i)}_s(X)$, $i=0,1,\ldots,n$, such that 
$P_j(X)=R^{(0)}_j(X)$, $Q_j(X)=R^{(n)}_j(X)$ for $j=1,\ldots,s$ and the decompositions
$R^{(i-1)}_1(X)\circ\ldots\circ R^{(i-1)}_s(X)$, $R^{(i)}_1(X)\circ\ldots\circ R^{(i)}_s(X)$ are
elementary related for any $i=1,\ldots,n$.
\end{theorem}
These two Theorems of J.~F.~Ritt give a presentation of the monoid $(\mathbb{C}[X],\circ)$.

\section{A complete presentation}

We follow G. Bergman paper \cite{B}: a \textit{complete presentation} is a presentation in which there
is a total order on the set of generators; this order is extended lexicographically on the set of 
words. A relation $W_1=W_2$ in which $W_1>W_2$ becomes $W_1\succ W_2$: any word containing $W_1$ as a 
sub-word, $W=AW_1B$, should be replaced by $W'=AW_2B$. The presentation is \textit{complete} if, 
starting with an arbitrary word and applying changes $W_1\succ W_2$ given by the presentation, we 
obtain, after finally many steps, a word which cannot be 'reduced'; this final result does not depend 
on the order of the reductions $W_1\succ W_2$ (we say that all \textit{ambiguities are solvable}). 

In the monoid $(\mathbb{C}[X],\circ)$ it will be sufficient to order the fundamental generators: 
we say that ${\bf A}(X)>{\bf B}(X)$ if ${\rm deg}{\bf A}(X)>{\rm deg}{\bf B}(X)$. The Ritt elementary
relations should be written as follows:
$$ \begin{array}{l} 
{\bf M}_p\circ{\bf M}_q\succ{\bf M}_q\circ{\bf M}_p,\mbox{ where }p>q\mbox{ are prime numbers};       \\
{\bf T}_p\circ{\bf T}_q\succ{\bf T}_q\circ{\bf T}_p,\mbox{ where }p>q\geq 3\mbox{ are prime numbers}; \\
{\bf R}_\lambda^{p,s}\circ{\bf M}_p\succ{\bf M}_p\circ{\bf R}_\rho^{p,s}\mbox{ where }p\mbox{ is a prime number}.
\end{array} $$
\begin{remark}
The polynomials ${\bf M}_2$ and ${\bf T}_2$ are in the same $G$-orbit. The
relation ${\bf T}_p\circ{\bf T}_2\succ{\bf T}_2\circ{\bf T}_p$ is a particular case of the general
relation  ${\bf R}_\lambda^{2,1}\circ{\bf M}_2\succ{\bf M}_2\circ{\bf R}_\rho^{2,1}$.
\end{remark} 
The ambiguities are, up to a $G$-equivalence, the following:
$$ \begin{array}{ll} 
a) & {\bf M}_p\circ{\bf M}_q\circ{\bf M}_r,\mbox{ where }p>q>r\mbox{ are prime numbers};        \\
b) & {\bf T}_p\circ{\bf T}_q\circ{\bf T}_r,\mbox{ where }p>q>r\geq 3\mbox{ are prime numbers};  \\
c) & {\bf R}_\lambda^{p,s}\circ{\bf M}_p\circ{\bf M}_q,\mbox{ where }p>q\mbox{ are prime numbers}.
\end{array} $$
In the following computations we underline the sub-word to be replaced. The first two ambiguities 
are solvable:
$$ \begin{array}{l} 
\underline{(\alpha_1{\bf M}_p\alpha_2^{-1})\circ(\alpha_2{\bf M}_q\alpha_3^{-1})}\circ(\alpha_3
                                                                             {\bf M}_r\alpha_4^{-1})\succ \\
\quad\quad\succ(\alpha_1{\bf M}_q\beta_1^{-1})\circ\underline{(\beta_1{\bf M}_p\alpha_3^{-1})\circ
                                                                   (\alpha_3{\bf M}_r\alpha_4^{-1})}\succ \\
\quad\quad\succ\underline{(\alpha_1{\bf M}_q\beta_1^{-1})\circ(\beta_1{\bf M}_r\beta_2^{-1})}\circ
                                                                     (\beta_2{\bf M}_p\alpha_4^{-1})\succ \\
\quad\quad\succ(\alpha_1{\bf M}_r\beta_3^{-1})\circ(\beta_3{\bf M}_q\beta_2^{-1})\circ(\beta_2
                                                                  {\bf M}_p\alpha_4^{-1})\mbox{ and }     \\ 
(\alpha_1{\bf M}_p\alpha_2^{-1})\circ\underline{(\alpha_2{\bf M}_q\alpha_3^{-1})\circ(\alpha_3
                                                                            {\bf M}_r\alpha_4^{-1})}\succ \\
\quad\quad\succ\underline{(\alpha_1{\bf M}_p\alpha_2^{-1})\circ(\alpha_2{\bf M}_r\gamma_1^{-1})}\circ
                                                                    (\gamma_1{\bf M}_q\alpha_4^{-1})\succ \\
\quad\quad\succ(\alpha_1{\bf M}_r\gamma_2^{-1})\circ\underline{(\gamma_2{\bf M}_p\gamma_1^{-1})\circ
                                                                   (\gamma_1{\bf M}_q\alpha_4^{-1})}\succ \\
\quad\quad\succ(\alpha_1{\bf M}_r\gamma_2^{-1})\circ(\gamma_2{\bf M}_q\gamma_3^{-1})\circ(\gamma_3
                                                                                 {\bf M}_p\alpha_4^{-1})                                                                                        
\end{array} $$
and the two results coincide. Similarly, the second ambiguity is solvable. The third ambiguity is 
solvable when ${\bf R}_\lambda^{p,s}=X^sG^{pq}(X)$: with the local notation 
$$ {\bf R}_\lambda^{pq,s}=X^sG^{pq}(X),{\bf R}_\rho^{pq,s}=X^sG(X^{pq}),{\bf R}_\mu^{pq,s}=X^sG^q(X^p),  
            {\bf R}_\nu^{pq,s}=X^sG^p(X^q)              $$
(where $p>q$ are prime numbers), we get two $G$-equivalent polynomials            
$$ \begin{array}{l} 
\underline{(\alpha_1{\bf R}_\lambda^{pq,s}\alpha_2^{-1})\circ(\alpha_2{\bf M}_p\alpha_3^{-1})}\circ(\alpha_3
                                                                                {\bf M}_q\alpha_4^{-1})\succ \\
\quad\quad\succ(\alpha_1{\bf M}_p\beta_1^{-1})\circ\underline{(\beta_1{\bf R}_\mu^{pq,s}\alpha_3^{-1})
                                                                 \circ(\alpha_3{\bf M}_q\alpha_4^{-1})}\succ \\
\quad\quad\succ\underline{(\alpha_1{\bf M}_p\beta_1^{-1})\circ(\beta_1{\bf M}_q\beta_2^{-1})}\circ(\beta_2
                                                                      {\bf R}_\rho^{pq,s}\alpha_4^{-1})\succ \\
\quad\quad\succ(\alpha_1{\bf M}_q\beta_3^{-1})\circ(\beta_3{\bf M}_p\beta_2^{-1})
                                                 \circ(\beta_2{\bf R}_\rho^{pq,s}\alpha_4^{-1})\mbox{ and }  \\ 
(\alpha_1{\bf R}_\lambda^{pq,s}\alpha_2^{-1})\circ\underline{(\alpha_2{\bf M}_p\alpha_3^{-1})\circ(\alpha_3
                                                                               {\bf M}_q\alpha_4^{-1})}\succ \\
\quad\quad\succ\underline{(\alpha_1{\bf R}_\lambda^{pq,s}\alpha_2^{-1})\circ(\alpha_2{\bf M}_q
                                                   \gamma_1^{-1})}\circ(\gamma_1{\bf M}_p\alpha_4^{-1})\succ \\
\quad\quad\succ(\alpha_1{\bf M}_q\gamma_2^{-1})\circ\underline{(\gamma_2{\bf R}_\nu^{pq,s}\gamma_1^{-1})\circ
                                                                      (\gamma_1{\bf M}_p\alpha_4^{-1})}\succ \\
\quad\quad\succ(\alpha_1{\bf M}_q\gamma_2^{-1})\circ(\gamma_2{\bf M}_p\gamma_3^{-1})\circ
                                                                    (\gamma_3{\bf R}_\rho^{pq,s}\alpha_4^{-1}).                                                                                        
\end{array} $$            
In the general case we have
$$ \begin{array}{l} 
\underline{(\alpha_1{\bf R}_\lambda^{p,s}\alpha_2^{-1})\circ(\alpha_2{\bf M}_p\alpha_3^{-1})}\circ(\alpha_3
                                                                                {\bf M}_q\alpha_4^{-1})\succ \\
\quad\quad\succ(\alpha_1{\bf M}_p\beta_1^{-1})\circ(\beta_1{\bf R}_\rho^{p,s}\alpha_3^{-1})
                                                          \circ(\alpha_3{\bf M}_q\alpha_4^{-1})\mbox{ and }  \\ 
(\alpha_1{\bf R}_\lambda^{p,s}\alpha_2^{-1})\circ\underline{(\alpha_2{\bf M}_p\alpha_3^{-1})\circ(\alpha_3
                                                                               {\bf M}_q\alpha_4^{-1})}\succ \\
\quad\quad\succ(\alpha_1{\bf R}_\lambda^{p,s}\alpha_2^{-1})\circ(\alpha_2{\bf M}_q
                                                        \gamma_1^{-1})\circ(\gamma_1{\bf M}_p\alpha_4^{-1}), \\
\end{array} $$   
hence we obtain a new relation (where $p>q$)
$$ (\alpha_1{\bf R}_\lambda^{p,s}\alpha_2^{-1})\circ(\alpha_2{\bf M}_q\gamma_1^{-1})\circ(\gamma_1{\bf M}_p
    \alpha_4^{-1})\succ(\alpha_1{\bf M}_p\beta_1^{-1})\circ(\beta_1{\bf R}_\rho^{p,s}\alpha_3^{-1})
                                                              \circ(\alpha_3{\bf M}_q\alpha_4^{-1}).     $$
The new ambiguity $ (\alpha_1{\bf R}_\lambda^{p,s}\alpha_2^{-1})\circ(\alpha_2{\bf M}_q\alpha_3^{-1})\circ
(\alpha_3{\bf M}_p\alpha_4^{-1})\circ(\alpha_4{\bf M}_r\alpha_5^{-1})$, where $p>q,r$, will give a new
relation (here $t={\rm min}(q,r)$ and $u={\rm max}(q,r)$)
$$ \begin{array}{l}
(\alpha_1{\bf R}_\lambda^{p,s}\alpha_2^{-1})\circ(\alpha_2{\bf M}_t\alpha_3^{-1})\circ(\alpha_3{\bf M}_u
                                       \alpha_4^{-1})\circ(\alpha_4{\bf M}_p\alpha_5^{-1})\succ  \\
    \quad\quad\succ(\alpha_1{\bf M}_p\beta_1^{-1})\circ(\beta_1{\bf R}_\rho^{p,s}\beta_2^{-1})\circ(\beta_2
                                                {\bf M}_t\beta_3^{-1})\circ(\beta_3{\bf M}_u\alpha_5^{-1}).    
\end{array} $$
New ambiguities (and new relations) appear.  \\
\emph{Proof of Theorem 1.3}. Suppose that all relations are of the form 
(here we take prime numbers $p>q_s\geq q_{s-1}\geq\ldots\geq q_1$)
$$ \begin{array}{l}
(\alpha_1{\bf R}_\lambda^{p,s}\beta_1^{-1})\circ(\beta_1{\bf M}_{q_1}\beta_2^{-1})\circ\ldots\circ
                    (\beta_s{\bf M}_{q_s}\beta_{s+1}^{-1})\circ(\beta_{s+1}{\bf M}_p\alpha_2^{-1})\succ  \\
    \quad\quad\succ(\alpha_1{\bf M}_p\gamma_1^{-1})\circ(\gamma_1{\bf R}_\rho^{p,s}\gamma_2^{-1})\circ
    (\gamma_2{\bf M}_{q_1}\gamma_3^{-1})\circ\ldots\circ(\gamma_{s+1}{\bf M}_{q_s}\alpha_2^{-1}).    
\end{array} $$
Now the ambiguity ($p>q_s\geq q_{s-1}\geq\ldots\geq q_{i+1}>r\geq q_i\geq\ldots\geq q_1$)
$$ (\alpha_1{\bf R}_\lambda^{p,s}\beta_1^{-1})\circ(\beta_1{\bf M}_{q_1}\beta_2^{-1})\circ\ldots\circ
   (\beta_s{\bf M}_{q_s}\beta_{s+1}^{-1})\circ(\beta_{s+1}{\bf M}_p\alpha_2^{-1})\circ(\alpha_2
                                                                              {\bf M}_r\alpha_3^{-1}) $$
is solvable:
$$ \begin{array}{l}
\underline{(\alpha_1{\bf R}_\lambda^{p,s}\beta_1^{-1})\circ(\beta_1{\bf M}_{q_1}\beta_2^{-1})\circ\ldots
   \circ(\beta_s{\bf M}_{q_s}\beta_{s+1}^{-1})\circ(\beta_{s+1}{\bf M}_p\alpha_2^{-1})}\circ(\alpha_2
                                                                          {\bf M}_r\alpha_3^{-1})\succ   \\ 
 \quad\succ(\alpha_1{\bf M}_p\gamma_1^{-1})\circ(\gamma_1{\bf R}_\rho^{p,s}\gamma_2^{-1})\circ
    (\gamma_2{\bf M}_{q_1}\gamma_3^{-1})\circ\ldots\circ\underline{(\gamma_{s+1}{\bf M}_{q_s}\alpha_2^{-1})
                                                           \circ(\alpha_2{\bf M}_r\alpha_3^{-1})}\succ   \\ 
 \quad\succ(\alpha_1{\bf M}_p\gamma_1^{-1})\circ(\gamma_1{\bf R}_\rho^{p,s}\gamma_2^{-1})\circ
    (\gamma_2{\bf M}_{q_1}\gamma_3^{-1})\circ\ldots\underline{\circ(\gamma_{s+1}{\bf M}_r\delta_s^{-1})}
                                                        \circ(\delta_s{\bf M}_{q_s}\alpha_3^{-1})\succ   \\                                                           
 \quad\succ\ldots\succ(\alpha_1{\bf M}_p\gamma_1^{-1})\circ(\gamma_1{\bf R}_\rho^{p,s}\gamma_2^{-1})\circ
                                                  (\gamma_2{\bf M}_{q_1}\gamma_3^{-1})\circ\ldots\circ   \\                                                           
 \quad\quad\quad\quad\circ (\gamma_{i+1}{\bf M}_{q_i}\gamma_{i+2}^{-1})\circ 
 (\gamma_{i+2}{\bf M}_r\delta_{i+1}^{-1})\circ(\delta_{i+1}{\bf M}_{q_{i+1}}\delta_{i+2}^{-1})\circ
                                            \ldots\circ(\delta_s{\bf M}_{q_s}\alpha_3^{-1})\mbox{ and }  \\
(\alpha_1{\bf R}_\lambda^{p,s}\beta_1^{-1})\circ(\beta_1{\bf M}_{q_1}\beta_2^{-1})\circ\ldots\circ
   (\beta_s{\bf M}_{q_s}\beta_{s+1}^{-1})\circ\underline{(\beta_{s+1}{\bf M}_p\alpha_2^{-1})\circ(\alpha_2
                                                                         {\bf M}_r\alpha_3^{-1})}\succ   \\
 \quad\succ(\alpha_1{\bf R}_\lambda^{p,s}\beta_1^{-1})\circ\ldots\circ
   \underline{(\beta_s{\bf M}_{q_s}\beta_{s+1}^{-1})\circ(\beta_{s+1}{\bf M}_r\eta_{s+1}^{-1})}\circ
                                                               (\eta_{s+1}{\bf M}_p\alpha_3^{-1})\succ   \\                                                                            
 \quad\succ\ldots\succ\underline{(\alpha_1{\bf R}_\lambda^{p,s}\beta_1^{-1})\circ\ldots\circ   
 (\beta_i{\bf M}_{q_i}\beta_{i+1}^{-1})\circ(\beta_{i+1}{\bf M}_r\eta_i^{-1})\circ(\eta_i{\bf M}_{q_i}
                                                                                \eta_{i+1}^{-1})\circ}   \\                                                                            
 \quad\quad\quad\quad\underline{\circ\ldots\circ(\eta_s{\bf M}_{q_s}\eta_{s+1}^{-1})\circ(\eta_{s+1}{\bf M}_p
                                                                                  \alpha_3^{-1})}\succ   \\                                                                                                                                                    
 \quad\succ(\alpha_1{\bf M}_p\eta_1^{-1})\circ(\eta_1{\bf R}_\rho^{p,s}\beta_1^{-1})\circ(\beta_1
                                                            {\bf M}_{q_1}\beta_2^{-1})\circ\ldots\circ   \\                                                                                                                                                    
 \quad\quad\quad\quad\circ(\beta_i{\bf M}_{q_i}\beta_{i+1}^{-1})\circ(\beta_{i+1}{\bf M}_r\eta_i^{-1})\circ
 (\eta_i{\bf M}_{q_{i+1}}\eta_{i+1})\circ\ldots\circ(\eta_s{\bf M}_{q_s}\alpha_3^{-1}).    
\end{array} $$ 
\hfill $\square$



\section{Appendix}

{\bf Decomposable polynomials of degree at most 12.} We classify decomposable $\delta$-polynomials of degree 
$d\leq 12$ and we give al their $\mu\nu\delta$-decompositions. We use notation 
$\mathcal{D}^{q,r\mid s,t;\Delta}=\mathcal{D}^{q,r;\Delta}\cap \mathcal{D}^{s,t;\Delta}$ for a double 
decomposition; lower indices, if any, stand for the parameters in a given family of $G$ orbits.  
\begin{center}
\begin{tabular}{|c|l |}
\hline
\hline
         & \quad \quad \quad \quad \quad \quad \quad \quad \quad \quad  Table 1: $ d=4 $ \quad \quad \, 
           \quad\quad \quad \quad \quad \quad \quad \quad \quad \quad \quad \quad \quad \quad \quad \quad \\
\hline
\hline 
$e$      & $ P(X)=X^4+4X^2+\rho X $                                                                       \\  
\hline
$0$      & $ \mathcal{D}^{2,2;4}:\, X^4=X^2\circ X^2 $                                                    \\
\hline  
$2$      & $ \mathcal{D}^{2,2;2}:\, X^4+4X^2=(X^2+4X)\circ X^2 $                                          \\
\hline  
\hline
\end{tabular}
\end{center}
\begin{center}
\begin{tabular}{|c|l |}
\hline
\hline
         & \quad \quad \quad \quad \quad \quad \quad \quad \quad \quad  Table 2: $ d=6 $                    \\
\hline
\hline
$e$   & $ P(X)=X^6+6X^4+\eta X^3+\varepsilon X^2+\rho X $                                                   \\
\hline
$0$   & $ \mathcal{D}^{2,3|3,2;6}:\, X^6=X^2\circ X^3=X^3\circ X^2 $                                        \\
\hline
$2$   & $ \mathcal{D}^{3,2;4}:\, X^6+6X^2=(X^3+6X)\circ X^2 $                                               \\
\hline 
$3$   & $ \mathcal{D}^{2,3;3}:\, X^6+6X^3=(X^2+6X)\circ X^3 $                                               \\
\hline
$4$   & $ \mathcal{D}^{2,3;2}_{\eta}:\, X^6+6X^4+\eta X^3+9X^2+3\eta X=(X^2+\eta X)\circ (X^3+3X) $         \\  
      & $ \mathcal{D}^{3,2;2}_{\varepsilon}:\, X^6+6X^4+\varepsilon X^2=(X^3+6X^2+\varepsilon X)\circ X^2 $ \\
      & $ \mathcal{D}^{2,3\mid 3,2;2}:\,X^6+6X^4+9X^2=X^2\circ (X^3+3X)=(X^3+6X^2+9X)\circ X^2 $\quad\quad  \\
\hline
\hline
\end{tabular}
\end{center}
\begin{center}
\begin{tabular}{|c|l |}
\hline
\hline
         & \quad \quad \quad \quad \quad \quad \quad \quad \quad \quad Table 3: $d=8$                      \\
\hline
\hline         
$e$  & $ P(X)=X^8+8X^6+\tau X^5+\theta X^4+\eta X^3+\varepsilon X^2+\rho X $                               \\    
\hline
$0$  & $ \mathcal{D}^{2,2,2;8}:\, X^8=X^2\circ X^2\circ X^2 $                                              \\
\hline
$2$  & $ \mathcal{D}^{4,2;6}:\,X^8+8X^2=(X^4+4X)\circ X^2  $                                               \\
\hline
$4$  & $ \mathcal{D}^{4,2;4}_{\varepsilon}:\, X^8+8X^4+\varepsilon X^2=(X^4+8X^2+\varepsilon X)\circ X^2 $ \\
     & $ \mathcal{D}^{2,2,2;4}:\, X^8+8X^4=(X^2+8X)\circ X^2\circ X^2 $                                    \\
\hline
$5$  & $ \mathcal{D}^{2,4;3}_{\theta}:\, X^8+8X^5+\theta X^4+16X^2+4\theta X=(X^2+\theta X)\circ(X^4+4X) $ \\
\hline
$6$  & $ \mathcal{D}^{2,4;2}_{\tau,\theta}:\, X^8+8X^6+\tau X^5+\theta X^4+4\tau X^3+(\frac{1}{4}\tau^2+
                           4\theta -64)X^2+(\frac{1}{2}\tau\theta-8\tau)X= $                               \\
     & \quad \quad \quad \quad \quad  $ =[X^2+(\theta-16)X]\circ (X^4+4X^2+\frac{1}{2}\tau X) $             \\                              
     & $ \mathcal{D}^{4,2;2}_{\theta,\varepsilon}:\, X^8+8X^6+\theta X^4+\varepsilon X^2=(X^4+8X^3+
                                         \theta X^2+\varepsilon X)\circ X^2 $                              \\
     & $ \mathcal{D}^{2,2,2;2}_{\theta}:\, X^8+8X^6+\theta X^4+(4\theta-64)X^2= $                          \\  
     & \quad \quad \quad \quad \quad $ =[X^2+(\theta-16)X]\circ (X^2+4X)\circ X^2 $                        \\
\hline
\hline
\end{tabular}
\end{center}
\begin{center}
\begin{tabular}{|c|l |}
\hline
\hline
        &  \quad \quad \quad \quad \quad \quad \quad \quad \quad \quad Table 5: $d=10$                    \\
\hline
\hline     
$e$     & $ P(X)=X^{10}+10X^8+\sigma X^7+\nu X^6+\tau X^5+\theta X^4+\eta X^3+\varepsilon X^2+\rho X $    \\ 
\hline
$0$     & $ \mathcal{D}^{2,5\mid 5,2;10}:\, X^{10}=X^2\circ X^5=X^5\circ X^2 $                                     \\
\hline
$2$     & $ \mathcal{D}^{5,2;8}:\, X^{10}+10X^2=(X^5+10X)\circ X^2 $                                               \\
\hline
$4$     & $ \mathcal{D}^{5,2;6}_{\varepsilon}:\, X^{10}+10X^4+\varepsilon X^2=(X^5+10X^2+\varepsilon X)\circ X^2 $ \\
\hline
$5$     & $ \mathcal{D}^{2,5;5}:\, X^{10}+10X^5=(X^2+10X)\circ X^5 $                                               \\
\hline
$6$     & $ \mathcal{D}^{2,5;4}_{\tau}:\, X^{10}+10X^6+\tau X^5+25X^2+5\tau X=(X^2+\tau X)\circ (X^5+5X) $         \\
        & $ \mathcal{D}^{5,2;4}_{\theta,\varepsilon}:\, X^{10}+10X^6+\theta X^4+\varepsilon X^2=(X^5+10X^3+
                                                          \theta X^2+\varepsilon X)\circ X^2 $                     \\
        & $ \mathcal{D}^{2,5\mid 5,2;4}:\, X^{10}+10X^6+25X^2=X^2\circ (X^5+5X)=(X^5+10X^3+25X)\circ X^2 $         \\
\hline
$7$     & $ \mathcal{D}^{2,5;3}_{\nu,\tau}:\, X^{10}+10X^7+\nu X^6+\tau X^5+25X^4+5\nu X^3+(\frac{1}{4}\nu^2+5\tau)X^2+
                                                                 \frac{1}{2}\nu\tau X= $                           \\
        & \quad \quad \quad \quad \quad  $ = (X^2+\tau X)\circ (X^5+5X^2+\frac{1}{2}\nu X) $                       \\
\hline
$8$     & $ \mathcal{D}^{2,5;2}_{\sigma,\nu,\tau}:\, X^{10}+10X^8+\sigma X^7+\nu X^6+\tau X^5+(\frac{1}{4}\sigma^2+
                                                 5\nu-125)X^4+$ \\
        & \quad $ +(\frac{1}{2}\sigma\nu-\frac{75}{2}\sigma+5\tau)X^3+(\frac{1}{4}\nu^2+
              \frac{1}{2}\sigma\tau-\frac{5}{2}\sigma^2-\frac{25}{2}\nu+\frac{625}{4})X^2+    $                    \\
        & \quad $   +(\frac{1}{2}\nu\tau-\frac{5}{2}\sigma\tau+\frac{125}{2}\sigma-\frac{25}{2}\tau)X= $           \\       
        & \quad \quad \quad \quad \quad $ =[X^2+(\tau-5\sigma)X]\circ [X^5+5X^3+
                                     \frac{1}{2}\sigma X^2+(\frac{1}{2}\nu-\frac{25}{2})X] $                      \\
        & $ \mathcal{D}^{5,2;2}_{\nu,\theta,\varepsilon}:\, X^{10}+10X^8+\nu X^6+\theta X^4+\varepsilon X^2= $     \\                                         
        & \quad \quad \quad \quad \quad  $ =(X^5+10X^4+\nu X^3+\theta X^2+
                                                                       \varepsilon X)\circ X^2 $                   \\                                         
        & $ \mathcal{D}^{2,5\mid 5,2;2}_{\nu}:\, X^{10}+10X^8+\nu X^6+(5\nu-125)X^4+(\frac{1}{4}\nu^2-\frac{25}{2}\nu+
                                                                          \frac{625}{4})X^2= $                     \\
        & \quad \quad \quad \quad \quad  $ =X^2\circ [X^5+5X^3+(\frac{1}{2}\nu-\frac{25}{2})X]= $                  \\
        & \quad \quad \quad \quad \quad  $=[X^5+10X^4+\nu X^3+(5\nu-125)X^2+
                                 (\frac{1}{4}\nu^2-\frac{25}{2}\nu+\frac{625}{4}) X]\circ X^2 $                    \\                                                                                                          
\hline
\hline
\end{tabular}
\end{center}
\begin{center}
\begin{tabular}{|c|l |}
\hline
\hline
        &  \quad \quad \quad \quad \quad \quad \quad \quad \quad \quad Table 4: $d=9$                     \\
\hline
\hline        
$e$  & $  P(X)=X^9+9X^7+\nu X^6+\tau X^5+\theta X^4+\eta X^3+\varepsilon X^2+\rho X $                   \\
\hline          
$0$     & $ \mathcal{D}^{3,3;9}:\, X^9=X^3\circ X^3 $                                                                 \\
\hline  
$3$     & $ \mathcal{D}^{3,3;6}:\, X^9+9X^3=(X^3+9X)\circ X^3 $                                                       \\
\hline
$6$     & $ \mathcal{D}^{3,3;3}_{\eta}:\, X^9+9X^6+\eta X^3=(X^3+9X^2+\eta X)\circ X^3 $                              \\
\hline
$7$     & $ \mathcal{D}^{3,3;2}_{\nu,\eta}:\, X^9+9X^7+\nu X^6+27X^5+6\nu X^4+\eta X^3+9\nu X^2+(3\eta-81)X=$ 
                                                                                       \quad \quad \quad  \\
        & \quad \quad \quad \quad \quad  $ =[X^3+\nu X^2+(\eta-27)X]\circ (X^3+3X) $                      \\
\hline
\hline
\end{tabular}
\end{center}
\begin{center}
\begin{tabular}{|c|l |}
\hline
\hline
      &  \quad \quad \quad \quad \quad \quad \quad \quad \quad \quad Table 6: $d=12$                               \\
\hline
\hline     
$e$   & $ P(X)=X^{12}+10X^{10}+\lambda X^8+\sigma X^7+\nu X^6+\tau X^5+\theta X^4+\eta X^3+
                                                                                 \varepsilon X^2+\rho X $           \\
\hline                                                                                              
$0$   & $ \mathcal{D}^{2,2,3\mid 2,3,2\mid 3,2,2;12}:\, X^{12}=X^2\circ X^2\circ X^3=X^2\circ X^3\circ X^2=
                                                                                 X^3\circ X^2\circ X^2 $            \\       
\hline
$2$   & $ \mathcal{D}^{6,2;10}:\, X^{12}+12X^2=(X^6+12X)\circ X^2 $                                                 \\
\hline 
$3$   & $ \mathcal{D}^{4,3;9}:\, X^{12}+12X^3=(X^4+12X)\circ X^3 $                                                  \\
\hline
$4$   & $ \mathcal{D}^{6,2;8}_{\varepsilon}:\, X^{12}+12X^4+\varepsilon X^2=(X^6+12X^2+\varepsilon X)\circ X^2 $    \\
      & $ \mathcal{D}^{3,2,2;8}:\, X^{12}+12X^4=(X^3+12X)\circ X^2\circ X^2 $                                       \\
\hline
$6$   & $ \mathcal{D}^{4,3;6}_{\eta}:\, X^{12}+12X^6+\eta X^3=(X^4+12X^2+\eta X)\circ X^3 $                         \\
      & $ \mathcal{D}^{6,2;6}_{\theta,\varepsilon}:\, X^{12}+12X^6+\theta X^4+\varepsilon X^2=(X^6+12X^3+
                                                                  \theta X^2+\varepsilon X)\circ X^2 $              \\
      & $ \mathcal{D}^{2,2,3\mid 2,3,2;6}:\, X^{12}+12X^6=(X^2+12X)\circ X^2\circ X^3=(X^2+12X)\circ X^3\circ X^2 $ \\
\hline
$7$   & $ \mathcal{D}^{2,6;5}_{\nu}:\, X^{12}+12X^7+\nu X^6+36X^2+6\nu X=(X^2+\nu X)\circ (X^6+6X) $                \\
\hline
$8$   & $ \mathcal{D}^{2,6;4}_{\sigma,\nu}:\, X^{12}+12X^8+\sigma X^7+\nu X^6+36X^4+6\sigma X^3+(\frac{1}{4}\sigma^2
                                           +6\nu)X^2+\frac{1}{2}\sigma\nu X= $                                      \\
      & \quad \quad \quad \quad \quad  $ =(X^2+\nu X)\circ (X^6+6X^2+\frac{1}{2}\sigma X) $                         \\
      & $ \mathcal{D}^{6,2;4}_{\nu,\theta,\varepsilon}:\, X^{12}+12X^8+\nu X^6+\theta X^4+\varepsilon X^2= $        \\
      & \quad \quad \quad \quad \quad  $ =(X^6+12X^4+\nu X^3+\theta X^2+\varepsilon X)\circ X^2 $       \\
      & $ \mathcal{D}^{2,3,2;4}_{\nu}:\, X^{12}+12X^8+\nu X^6+36X^4+6\nu X^2=(X^2+\nu X)\circ (X^3+6X)\circ X^2 $   \\
      & $ \mathcal{D}^{3,2,2;4}_{\theta}:\, X^{12}+12X^8+\theta X^4=(X^3+12X^2+\theta X)\circ X^2\circ X^2  $       \\                   
      & $ \mathcal{D}^{2,3,2\mid 3,2,2;4}:\, X^{12}+12X^8+36X^4= $   \\
      & \quad \quad \quad \quad \quad  $=X^2\circ (X^3+6X)\circ X^2=(X^3+12X^2+36X)\circ X^2\circ X^2$              \\
\hline
$9$   & $ \mathcal{D}^{2,6;3}_{\mu,\sigma,\nu}:\, X^{12}+12X^9+\mu X^8+\sigma X^7+\nu X^6+6\mu X^5+(\frac{1}{4}\mu^2+
                                                                             6\sigma)X^4+ $                         \\
      & \quad $ +(\frac{\mu\sigma}{2}+6\nu -216)X^3+(\frac{1}{2}\mu\nu+\frac{1}{4}\sigma^2-18\mu)X^2+
                                                                    (\frac{1}{2}\sigma\nu-18\sigma)X= $             \\
      & \quad\quad\quad\quad\quad $ =[X^2+(\nu-36)X]\circ (X^6+6X^3+\frac{1}{2}\mu X^2+\frac{\sigma}{2}X) $         \\
      & $\mathcal{D}^{3,4;3}_{\mu,\theta}:\, X^{12}+12X^9+\mu X^8+48X^6+8\mu X^5+\theta X^4+64X^3+16\mu X^2+
                                                                                                4\theta X=$         \\
      & \quad\quad \quad \quad \quad $ =(X^3+\mu X^2+\theta X)\circ (X^4+4X)  $                                     \\
      & $ \mathcal{D}^{4,3;3}_{\nu,\eta}:\, X^{12}+12X^9+\nu X^6+\eta X^3=(X^4+12X^3+\nu X^2+\eta X)\circ X^3  $    \\
    
      & $ \mathcal{D}^{2,2,3;3}_{\nu}:\, X^{12}+12X^9+\nu X^6+(6\nu-216)X^3= $                                      \\
      & \quad \quad \quad \quad \quad $ =[X^2+(\nu-36)X]\circ (X^2+6X)\circ X^3  $                                  \\
      & $ \mathcal{D}^{3,4\mid 4,3;3}:\, X^{12}+12X^9+48X^6+64X^3=  $                                               \\
      & \quad \quad \quad \quad \quad $ =X^3\circ (X^4+4X)=(X^4+12X^3+48X^2+64X)\circ X^3 $                         \\
\hline
\end{tabular}
\end{center}
\begin{center}
\begin{tabular}{|c|l |}
\hline
$10$ & $ \mathcal{D}^{2,6;2}_{\lambda,\mu,\sigma,\nu}:\, X^{12}+12X^{10}+\lambda X^9+\mu X^8+\sigma X^7+\nu X^6+
                                                    (\frac{1}{2}\lambda\mu-54\lambda+6\sigma)X^5+ $       \\
     & \quad $ +(-\frac{9}{2}\lambda^2+\frac{1}{2}\lambda\sigma+\frac{1}{4}\mu^2-54\mu+6\nu+1620)X^4+ $   \\
     & \quad $ +(-\frac{1}{8}\lambda^3-6\lambda\mu+\frac{1}{2}\lambda\nu+\frac{1}{2}\mu\sigma+216\lambda-
                                                                                        18\sigma)X^3+$    \\
     & \quad $ +(-\frac{1}{8}\lambda^2\mu+\frac{27}{2}\lambda^2-3\lambda\sigma-3\mu^2+\frac{1}{2}\mu\nu+
                                                         \frac{1}{4}\sigma^2+216\mu-18\nu-3888)X^2+ $     \\ 
     & \quad $ +(\frac{3}{4}\lambda^3-\frac{1}{8}\lambda^2\sigma+18\lambda
              \mu-3\lambda\nu-3\mu\sigma+\frac{1}{2}\sigma\nu-648\lambda+108\sigma)X= $                   \\   
     & \quad \quad \quad \quad \quad $=[X^2+(-\frac{1}{4}\lambda^2-6\mu+\nu+216)X]\circ   $               \\ 
     & \quad \quad \quad \quad \quad $  \circ [X^6+6X^4+\frac{1}{2}\lambda X^3+(\frac{1}{2}\mu-18)X^2+
                                                                 (-3\lambda+\frac{1}{2}\sigma)X]  $       \\                                               
     & $ \mathcal{D}^{3,4;2}_{\lambda,\mu,\theta}:\, X^{12}+12X^{10}+\lambda X^9+\mu X^8+8\lambda X^7+(\frac{1}{3}
                          \lambda^2+8\mu-320)X^6+ $       \\                                                            
     & \quad $ +(\frac{2}{3}\lambda\mu-16\lambda)X^5+\theta X^4+(\frac{1}{27}\lambda^3+\frac{8}{3}\lambda                           
                                                                           \mu-128\lambda)X^3+    $       \\                                               
     & \quad $ +(\frac{1}{9}\lambda^2\mu-\frac{32}{3}\lambda^2-64\mu+4\theta+3072)X^2  
               +(-\frac{4}{9}\lambda^3-\frac{16}{3}\lambda\mu+\frac{1}{3}\lambda\theta+256\lambda)X= $    \\                                   
     & \quad \quad \quad \quad \quad $ =[X^3+(\mu-48)X^2+(-\frac{4}{3}\lambda^2-16\mu+\theta+768)X]\circ $\\                                                                 
     & \quad \quad \quad \quad \quad $ \circ (X^4+4X^2+\frac{1}{3}\lambda X)  $                           \\                                                                 
     & $ \mathcal{D}^{4,3;2}_{\lambda,\sigma,\eta}:\, X^{12}+12X^{10}+\lambda X^9+54X^8+9\lambda X^7+\nu X^6+
                                                                                   27\lambda X^5+ $       \\
     & \quad $ +(6\nu-567)X^4+\eta X^3+(9\nu-972)X^2+(3\eta-81\lambda)X= $                                \\
     & \quad \quad \quad \quad \quad $=[X^4+\lambda X^3+(\nu-108)X^2+(-27\lambda+\eta)X]\circ (X^3+3X) $  \\
     & $ \mathcal{D}^{6,2;2}_{\mu,\nu,\theta,\varepsilon}:\, X^{12}+12X^{10}+\mu X^8+\nu X^6+\theta X^4+
                                                                                      \varepsilon X^2= $  \\
     & \quad \quad \quad \quad \quad $ =(X^6+12X^5+\mu X^4+\nu X^3+\theta X^2+\varepsilon X)\circ X^2 $   \\
     & $ \mathcal{D}^{2,2,3;2}_{\lambda,\nu}:\, X^{12}+12X^{10}+\lambda X^9+54X^8+9\lambda X^7+\nu X^6+27
     \lambda X^5+$\\                                                                                        
     & \quad $ +(6\nu-567)X^4+(-\frac{1}{8}\lambda^3+\frac{1}{2}\lambda\nu-27\lambda)X^3+         $       \\
     & \quad $ +(9\nu-972)X^2+(-\frac{3}{8}\lambda^3+\frac{3}{2}\lambda\nu-162\lambda)X= $                \\                                                                             
     & \quad \quad \quad \quad \quad $=[X^2+(-\frac{1}{4}\lambda^2+\nu-108)X]\circ (X^2+
                                                                   \frac{1}{2}\lambda X)\circ (X^3+3X) $  \\
     & $ \mathcal{D}^{2,3,2;2}_{\mu,\nu}:\, X^{12}+12X^{10}+\mu X^8+\nu X^6+(\frac{1}{4}\mu^2-54\mu+
     6\nu+1620)X^4+ $  \\                                                                                        
     & \quad $ +(-3\mu^2+\frac{1}{2}\mu\nu+216\mu-18\nu-3888)X^2=         $                               \\
     & \quad\quad\quad\quad\quad $=[X^2+(-6\mu+\nu+216)X]\circ [X^3+6X^2+(\frac{1}{2}\mu-18)X]\circ X^2 $ \\                                                              
     & $ \mathcal{D}^{3,2,2;2}_{\mu,\theta}:\, X^{12}+12X^{10}+\mu X^8+(8\mu-320)X^6+\theta X^4+    $     \\                                                                                                                                                              
     & \quad $ +(-64\mu+4\theta+3072)X^2= $                                                               \\                                                                                                                                                              
     & \quad \quad \quad \quad \quad $=[X^3+(\mu-48)X^2+(-16\mu+\theta+768)X]\circ (X^2+4X)\circ X^2   $  \\                                                                       
     & $ \mathcal{D}^{2,2,3\mid 2,3,2;2}_{\nu}:\, X^{12}+12X^{10}+54X^8+\nu X^6+(6\nu-567)X^4+
                                                                                  (9\nu-972)X^2=    $     \\                                                                                                                                                              
     & \quad \quad \quad \quad \quad $=[X^2+(\nu-108)X]\circ X^2\circ (X^3+3X)=   $                       \\
     & \quad \quad \quad \quad \quad $=[X^2+(\nu-108)X]\circ (X^3+6X^2+9X)\circ X^2   $                   \\                                                                                                                                              
     & $ \mathcal{D}^{2,3,2\mid 3,2,2;2}_{\mu}:\, X^{12}+12X^{10}+\mu X^8+(8\mu-320)X^6+(\frac{1}{4}\mu^2-
                                                                                         6\mu-300)X^4+$   \\                                                                                                                                              
     & \quad $ +(\mu^2-88\mu+1872)X^2=    $                                                               \\                                                                                                                                                              
     & \quad \quad \quad \quad \quad $=[X^2+(2\mu-104)X]\circ [X^3+6X^2+(\frac{1}{2}\mu-18)X]\circ X^2= $ \\
     & \quad \quad \quad \quad \quad $=[X^3+(\mu-48)X^2+(\frac{1}{4}\mu^2-22\mu+468)X]\circ (X^2+4X)
                                                                          \circ X^2   $                   \\                                                                                                                                              
     & $ \mathcal{D}^{2,2,3\mid 2,3,2\mid 3,2,2;2}:\, X^{12}+12X^{10}+54X^8+112X^6+105X^4+36X^2=    $     \\                                                                                                                                                              
     & \quad\quad\quad\quad\quad$=(X^2+4X)\circ X^2\circ (X^3+3X)=                   $                    \\                                                                                                                                                              
     & \quad \quad \quad \quad \quad $ =(X^2+4X)\circ (X^3+6X^2+9X)\circ X^2=   $                         \\
     & \quad \quad \quad \quad \quad $ =(X^3+6X^2+9X)\circ (X^2+4X) \circ X^2   $                         \\                                                                                                                                                                                                                   
\hline
\hline       
\end{tabular}
\end{center}

\newcommand{\arxiv}[1]
{\texttt{\href{http://arxiv.org/abs/#1}{arxiv:#1}}}
\newcommand{\arx}[1]
{\texttt{\href{http://arxiv.org/abs/#1}{arXiv:}}
\texttt{\href{http://arxiv.org/abs/#1}{#1}}}
\newcommand{\doi}[1]
{\texttt{\href{http://dx.doi.org/#1}{doi:#1}}}

\end{document}